\documentclass{amsart}
\usepackage{amsmath,amsthm,amsfonts,amssymb}
\usepackage{mathrsfs}
\usepackage{latexsym}
\usepackage[latin1]{inputenc}
\usepackage{bm}
\usepackage{amssymb}
\usepackage[all]{xy}
\usepackage{enumerate}
\usepackage{stmaryrd}
\usepackage{color}
\usepackage{mathrsfs}
\usepackage{url}
\usepackage{soul}
\usepackage{ dsfont }
\usepackage{cite}
\usepackage{hyperref}

\newcommand{\mO}{{\mathcal{O}}}

\newcommand{\mU}{{\mathcal{U}}}
\newcommand{\mV}{{\mathcal{V}}}

\newcommand{\PP}{\mathbb{P}}

\newcommand{\los}[2]{{\mathsf{#1}}\!\not{\!\uparrow}\,\mathsf{G}_{\operatorname{#2}}}
\newcommand{\sfin}{{\mathsf{S}_{\textnormal{fin}}}}
\newcommand{\sone}{{\mathsf{S}_{\textnormal{1}}}}
\newcommand{\playerone}{\mathsf{Player~I}}
\newcommand{\playertwo}{\mathsf{Player~II}}
\newcommand{\gfin}{\mathsf{G}_{\operatorname{fin}}}
\newcommand{\gone}{\mathsf{G}_{\operatorname{1}}}

\newcommand{\opV}{\mathsf{V}}
\newcommand{\Sz}{Szewczak}

\newtheorem{theorem}{Theorem}[section]
\newtheorem{lemma}[theorem]{Lemma}
\newtheorem{proposition}[theorem]{Proposition}

\newtheorem{example}[theorem]{Example}

\newtheorem{question}[theorem]{Question}

\begin{document}
\title{Pointless proofs of the Menger and Rothberger games}

\author[R. M. Mezabarba]{Renan M. Mezabarba}
\address{Centro de Ci\^encias Exatas,
Universidade Federal do Espírito Santo,
Vit\'oria, ES, 29075-910, Brazil}
\email{renan.mezabarba@ufes.br}


%


\begin{abstract}
We present proofs of Hurewicz's and Pawlikowski's Theorems concerning topological games, but in the context of lattices, in which one cannot argue with points.
\end{abstract}

\maketitle

\section*{Introduction}

The framework of topological games and selection principles, the latter introduced by Scheepers~\cite{Scheep1996}, have been a fruitful source for both questions and answers in the field, as can be seen in the surveys~\cite{sakaischeepers,tsabanextravaganza}. We refer the reader to the article~\cite{AURICHI2019305}, by Aurichi and Dias, which presents  an introduction to the subject, as well as the main notations we follow along this work.
 
In an attempt to extend some results regarding selective variations of tightness in the context of $C_p$-theory \cite{AURICHI2019187} to the realm of topological groups, the following two important theorems about open covers appeared to be relevant:

\begin{itemize}
\item \textbf{Hurewicz's Theorem~\cite{oldHur}.} A topological space satisfies $\sfin(\mO,\mO)$ if, and only if, $\playerone$ does not have a winning strategy in the game $\gfin(\mO,\mO)$;\medskip
\item \textbf{Pawlikowski's Theorem~\cite{Pawlikowski1994}.} A topological space satisfies $\sone(\mO,\mO)$ if, and only if, $\playerone$ does not have a winning strategy in the game $\gone(\mO,\mO)$.
\end{itemize}

Roughly speaking, our original goal was to settle the same equivalences stated above but replacing the family $\mO$, of open covers of a topological space, with the set $\Omega_x$ whose elements are those subsets of a topological group containing the point $x$ in their closures. Although Hurewicz's and Pawlikowski's Theorems are not directly related to this problem, their proofs could provide useful insights. This ended up taking us to the work of \Sz~and Tsaban~\cite{Tsaban}, in which the authors present conceptual proofs for both the results in such a way that the role of each hypothesis becomes very clear.

Even though the analysis made so far about their proofs has not helped in our original problem, we have come to an unwittingly disclosure: both proofs can be carried out in the topology of the space rather than in the space itself, meaning that the results are of an order-theoretic nature. In other words, Hurewicz's and Pawlikowski's Theorems can be stated and proved in the context of \emph{Pointless Topology}, in the sense of Johnstone~\cite{johnstone1983,Johnstone}. The purpose of this work is to present this generalization, based on the proofs in~\cite{Tsaban} and in a few basic concepts of pointless topology, which we shall recall along the next sections\footnote{For a more systematic treatment of pointless topology, we refer the reader to~\cite{frames,Johnstone}.}.

\section{The Menger game played on nice lattices}

Let $(\PP,\leq)$ be a \emph{lattice}, i.e., a poset such that for every $a,b\in\PP$ there are $\sup\{a,b\}$ and $\inf\{a,b\}$, usually denoted by $a\vee b$ and $a\wedge b$, respectively the \emph{join} and the \emph{meet} of $a$ and $b$. For a fixed element $p\in\PP$, let us denote by $\mathsf{V}_p$ the family of all subsets $A$ of $\PP$ such that $\sup A=p$.

The family $\opV_p$ may be thought as a generalization of the family of open covers of a topological space. Indeed, if $(X,\tau)$ is a topological space, then $(\tau,\subseteq)$ is a lattice\footnote{Actually, it is more than a lattice: it is a frame, as we shall recall soon.} such that $\mU\subseteq \tau$ is an open cover for $X$ if, and only if, $\sup\mU=X$. Thus, with the notation of the previous paragraph, the family $\mathcal{O}_X$ of the open covers of $X$ coincides with the family $\opV_X$. In this sense, the main task of this section is to discuss the proof of the following theorem, which is a generalization of Hurewicz's Theorem.



\begin{theorem}\label{main.theo1}
Let $\PP$ be a lattice with enough prime elements. For every element $p\in\PP$, $\sfin(\mathsf{V}_p,\mathsf{V}_p)$ holds if, and only if, $\los{I}{fin}(\mathsf{V}_p,\mathsf{V}_p)$ holds.
\end{theorem}

A \textbf{prime element} of the lattice $\PP$ is an element $q\in\PP$ (not equal to $1$ if $\PP$ is bounded), such that $a\wedge b\leq q$ implies $a\leq q$ or $b\leq q$ for every $a,b\in \PP$. The \emph{enoughness} means that if $a,b\in\PP$ satisfy $a\not\leq b$, then there exists a prime element $q\in\PP$ such that $b\leq q$ and $a\not\leq q$. These elements shall play the role of the (complement of) points in our adaptation of the arguments of \Sz~and Tsaban.

Let us begin by noticing that if $\sfin(\opV_p,\opV_p)$ holds, then for every $A\in\opV_p$ there exists a countable subset ${B\subseteq A}$ such that $\sup B=p$, i.e., $\PP$ is \emph{Lindelöf} in an order-theoretic sense.~Thus we may restrict $\playertwo$'s moves to countable subsets of $\PP$. On the other hand, since suprema are associative, we may assume $\playerone$'s moves are countable and increasing, in such a way that $\playertwo$ may select a single element per inning.

Indeed, if $\playerone$ plays a subset $A_n\in\opV_p$ in the inning $n$, then there exists a countable subset $\{a^n_m:m\in\omega\}\subseteq A_n$ such that $\sup_{m\in\omega}a_m^n=p$, implying that $B_n=\left\{b_m:m\in\omega\right\}\in\opV_p$ is countable and increasing, where $b_m=a_0^n\vee\dotso\vee a_m^n$ for each $m$. Now, if $\playertwo$ selects $b_{m_n}\in B_n$ for each $n$ such that $\sup_{n\in\omega}b_{m_n}=p$, the associativity of joins yields
\[p=\sup_{n\in\omega}b_{m_n}=\sup_{n\in\omega}\bigcup\left\{a_0^n,\dotso,a^n_{m_n}\right\},\]
showing that there is no loss of generality in our assumption. As a last remark, notice that we may suppose that an answer $A=\{a_n:n\in\omega\}\in\opV_p$ of $\playerone$ to some element $a\in\PP$ selected by $\playertwo$ is such that $a_0=a$, since we can replace $A$ with $\{a\vee a_n:n\in\omega\}\in\opV_p$.

With these assumptions, a strategy for $\playerone$ in the game $\gfin(\opV_p,\opV_p)$ may be identified with a family $\sigma=\{A_s:s\in \omega^{<\omega}\}$ of subsets of $\PP$ such that $A_s\in \opV_p$ for every finite sequence $s\in\omega^{<\omega}$, where $A_{\langle\,\rangle}=\{a_j:j\in\omega\}$ is $\playerone$'s first move, $A_{\langle m\rangle}=\{a_{m,j}:j\in\omega\}$ is $\playerone$'s answer to $a_m$, $A_{\langle m,n\rangle}=\{a_{m,n,j}:j\in\omega\}$ is $\playerone$'s answer to $a_{m,n}$ and so on. Let us say that such a strategy $\sigma$ is a \textbf{nice strategy} for $\playerone$ in the game $\gfin(\opV_p,\opV_p)$.

All the previous simplifications were made by \Sz~and Tsaban~\cite{Tsaban} in the context of open covers in order to reduce the proof of Hurewicz's Theorem to a tricky lemma about tail covers. Recall that a countable open cover $\mU$ of a topological space $X$ is called a \emph{tail cover} \cite{Tsaban} if the set of intersections of cofinite subsets of $\mU$ is an open cover for $X$. Here, the translation is straightforward: a subset $A\subseteq\PP$ is a \textbf{tail set} (with respect to $p\in\PP$) if $\sup A=p$ and for every cofinite subset $B$ of $A$ there exists $\inf B$ such that $\sup\{\inf B:B$ is a cofinite subset of $A\}=p$. The order-theoretic version of their lemma is the following.

\begin{lemma}
Let $\sigma=\{A_s:s\in \omega^{<\omega}\}$ be a nice strategy for $\playerone$ and for $n\in\omega$ let $B_{n+1}=\bigcup_{s\in \omega^n}A_s$. Then $B_{n+1}$ is a tail set with respect to $p$.
\end{lemma}

Note that $B_{n+1}$ is the union of all possible answers of $\playerone$ in the $n$-th inning. 

\begin{proof} We proceed by induction.

Since $B_1=A_{\langle\,\rangle}$ is countable and increasing, it follows easily that
\[\sup\left\{\inf{B}:B\text{ is a cofinite subset of $A_{\langle\,\rangle}$}\right\}=p,\]
showing that $B_1$ is a tail set. Now, we assume $B_n$ is a tail set for $n>1$. Since every $A_s$ is countable, we may write $B_n=\{b_n:n\in\omega\}$ and then we put
\[B_{n+1}=\bigcup_{j\in\omega}\left\{b^j_n:n\in\omega\right\},\]
where $\{b_m=b_0^m,b_1^m,b_2^m,\dotso\}$ is the (increasing) answer of $\playerone$ to $b_m\in B_n$. We shall show that $B_{n+1}$ is a tail set with respect to $p$.

Let $B\subseteq B_{n+1}$ be a cofinite subset of $B_{n+1}$ and for every $k\in\omega$ consider $m_k=\min\{n:b_n^k\in B\}$. Note that since $B$ is cofinite, $m_k=0$ for all but finitely many $k$, hence $C=\{k:m_k=0\}$ is cofinite. Now, the niceness of strategy $\sigma$ implies that $D=\{b_{m_k}^k:k\in C\}$ is also a cofinite subset of $B_n$, while the induction hypothesis guarantees the existence of $\inf D$. Finally, since $\inf B\cap\{b_n^k:n\in\omega\}=b^k_{m_k}$ for every $k\in\omega$, it follows that
\[\inf B=\inf\bigcup_{k\in\omega}B\cap\left\{b_n^k:n\in\omega\right\}=\inf_{k\in\omega}b_{m_k}^k=\inf D\wedge \bigwedge_{k\not\in C}b_{m_k}^k,\]
which existence follows because $\PP$ is a lattice. Now, we show that the supremum of those cofinal subsets is $p$.

Clearly we have $\inf B\leq p$ for all cofinite subsets $B\subseteq B_{n+1}$. So, we just need to show that if $q\in\PP$ is such that $\inf B\leq q$ for all cofinite subsets $B$ of $B_{n+1}$, then $p\leq q$. Suppose, contrary to our claim, that $p\not\leq q$. Since $\PP$ has enough prime elements, there exists a $\tilde{q}\in\PP$ such that $q\leq \tilde{q}$ and $p\not\leq \tilde{q}$. Now, the inductive hypothesis yields a cofinite subset $C\subseteq B_n$ such that $\inf C\not\leq \tilde{q}$, say $C=\{b_n:n\in J\}$ for a cofinite subset $J\subseteq \omega$. Since $\sup\{b_k^m:k\in\omega\}=p$ for all $m\in\omega$, for each $m\not\in J$ there exists an $n_m\in\omega$ such that $b_{n_m}^m\not\leq \tilde{q}$, from which it follows that $\bigwedge_{m\in\omega\setminus J}b_{n_m}^m\not\leq \tilde{q}$. Since the family
\[E:=\left\{b_m^n:n\in J,m\in\omega\right\}\cup\left\{b_m^k:m\in\omega\setminus J\text{ and }k\geq n_m\right\}\]
is a cofinite subset of $B_{n+1}$ such that $\inf E=\inf C\wedge \bigwedge_{m\in\omega\setminus J}b^m_{n_m}$, the primeness of $\tilde{q}$ gives $\inf E\not\leq \tilde{q}$. Therefore, $\inf E\not\leq q$.
\end{proof}

\begin{proof}[Proof of Theorem~\ref{main.theo1}]
The simplifications made so far allows us to assume $\playerone$ has a nice strategy $\sigma$ such that, with the same notations of the previous lemma, the family $B_{n+1}$ is a tail set for every $n\in\omega$. Consequently, \[V_{n+1}:=\{\inf B:B\subseteq B_{n+1}\text{ is cofinite}\}\in\opV_p.\] Now, $\sfin(\opV_p,\opV_p)$ applies to the sequence $(V_1,V_2,\dotso)$, giving a sequence $(F_1,F_2,\dotso)$ of finite subsets $F_n\subseteq V_n$ such that $\sup\bigcup_{n\in\omega}F_n=p$. Notice that for each $f\in F_n$ there exists a cofinite subset $B_f^n\subseteq B_{n+1}$ such that $f=\inf B_{f}^n$, and $C^n=\bigcap_{f\in F_n}B_f^n$ is a cofinite subset of $B_{n+1}$.

Finally, let us see how $\playertwo$ can defeat the nice strategy $\sigma$: in the $n$-th inning, $\playerone$ selects an infinite subset $A$ of $B_{n+1}$, which must intersect $C^n$, so $\playertwo$ may choose $f_n\in C^n\cap A$. Then 
\[p=\sup\{\inf B^n_f:f\in F_n,n\in\omega\setminus\{0\}\}\leq \sup_{n\in\omega}f_n\leq p,\]
showing that $\playertwo$ wins, as desired.
\end{proof}

\section{The Rothberger game played on nicer lattices}

We now proceed to discuss the adaptation of Pawlikowski's Theorem, which shall require stronger conditions on the lattice $\PP$. For our preliminary considerations, let us suppose that $\PP$ is a bounded lattice with enough prime elements. For simplicity, such lattices will be called \textbf{pre-Pawlikowski}.

As it happens in the original proof of Pawlikowski, we will need to use Hurewicz's theorem to provide nice plays in auxiliary games. At the beginning of the proof, one uses the simple exercise that a countable union of Menger spaces is again a Menger space. This still happens with pre-Pawlikowski lattices when we consider the right generalization of union.

\begin{lemma}
If for every $n\in\omega$ there is a pre-Pawlikowski lattice $\PP_n$ such that $\sfin(\opV_{1_{\PP_n}},\opV_{1_{\PP_n}})$ holds, then $\PP:=\prod_{n\in\omega}\PP_n$ is a pre-Pawlikowski lattice satisfying $\sfin(\opV_{1_\PP},\opV_{1_\PP})$, where $\PP$ is endowed with the pointwise ordering and $1_\PP:=(1_{\PP_n})_{n\in\omega}$.
\end{lemma}
\begin{proof}
The reader can easily show that $\PP$ is a pre-Pawlikowski lattice such that $1_\PP$ is its maximum element. Now, for a sequence $(A_n)_{n\in\omega}$ of subsets of $\PP$ such that $\sup A_n=1_\PP$ for each $n\in\omega$, we shall select a sequence $(F_n)_{n\in\omega}$ of subsets such that $F_n\in [A_n]^{<\omega}$ for every $n$ and $\sup\bigcup_{n\in\omega}F_n=1_\PP$.

Let $\{Q_n:n\in\omega\}$ be an infinite partition of $\omega$ such that each $Q_n$ is infinite. For a fixed $n\in\omega$ and $j\in Q_n$ we have $\sup A_j=1_\PP$, implying that $\sup \pi_n[A_j]=1_{\PP_n}$, where $\pi_n\colon \PP\to\PP_n$ is the obvious projection. Then, for each $j\in Q_n$ there exists a finite subset $F_j\subseteq A_j$ such that $\sup \bigcup_{j\in\omega}\pi_n[A_j]=1_{\PP_n}$, from which it easily follows that $\sup\bigcup_{n\in\omega}F_n=1_\PP$.
\end{proof}

With this lemma, we can show that if $\PP$ is a pre-Pawlikowski lattice satisfying $\sfin(\opV_{1},\opV_{1})$, then every strategy of $\playerone$ in the game $\gfin(\opV_1,\opV_1)$ can be severely defeated, in the following sense.

\begin{proposition}\label{strong.lose.prop}
Let $\PP$ be a pre-Pawlikowski lattice and let $\sigma$ be a strategy for $\playerone$ in the game $\gfin(\opV_1,\opV_1)$. If $\sfin(\opV_1,\opV_1)$ holds, then there exists a play $(A_0,F_0,A_1,F_1,\dotso)$ according to $\sigma$ such that for every $p\in\PP\setminus\{1\}$ we have $\sup F_n\not\leq p$ for infinitely many $n$.
\end{proposition}

\begin{proof}
The previous lemma guarantees that $\PP^\omega$ is a pre-Pawlikowski lattice satisfying $\sfin(\opV_{1_\omega},\opV_{1_\omega})$, where $1_\omega:=(1)_{n\in\omega}$. Now, we use $\sigma$ to cook up a strategy $\tilde{\sigma}$ for $\playerone$ in the game $\gfin(\opV_{1_\omega},\opV_{1_\omega})$ played on $\PP^\omega$. For a subset $A\subseteq \PP$, let $\tilde{A}:= \{\delta_n(a):a\in A$ and $n\in\omega\}$, where $\delta_n(a)\in\PP^\omega$ is such that $\delta_n(a)(m):=0$ for $m\ne n$ and $\delta_n(a)(n)=a$. It is easy to see that if $\sup A=1$, then $\sup\tilde{A}=1_\omega$.

We define $\tilde{\sigma}(\emptyset):=\tilde{A_0}$, where $\sigma(\emptyset):=A_0$. If $\playertwo$ selects a finite subset $\tilde{F_0}\subseteq \tilde{A_0}$, let $F_0:=\{a\in A_0:\delta_n(a)\in \tilde{F_0}$ for some $n\in\omega\}$, a finite subset of $A_0$, and then define $\tilde{\sigma}(\tilde{F_0}):=\tilde{A_1}$, where $A_1:=\sigma(F_0)$. Proceeding like this, we obtain a strategy $\tilde{\sigma}$ for $\playerone$ in the game $\gfin(\opV_{1_\omega},\opV_{1_\omega})$ played on $\PP^\omega$.

The lattice version of Hurewicz's theorem yields a play $(\tilde{A_0},\tilde{F_0},\tilde{A_1},\tilde{F_1},\dotso)$ according to $\tilde{\sigma}$ such that $\sup\bigcup_{n\in\omega}\tilde{F_n}=1_\omega$, and by the way we defined $\tilde{\sigma}$, this play corresponds to a play $(A_0,F_0,A_1,F_1,\dotso)$ according to $\sigma$. We claim this latter play has the desired property. Indeed, for a fixed $p\in\PP\setminus\{1\}$, we have $\tilde{p}_0:=(p,1,1,\dotso){<}\,1_\omega$, from which it follows that there exists $n_0\in\omega$ and $f:=(f_n)_{n\in\omega}\in\tilde{F}_{n_0}$ such that $f\not\leq \tilde{p}_{0}$, implying $f_0\in F_{n_0}$ is such that $f_0\not\leq p$, hence $\sup F_{n_0}\not\leq p$. Since $F:=\bigcup_{j\leq n_0}F_j$ is finite, the set $N:=\{n:\delta_n(a)\in F$ for some $a\}$ is also finite, which allow us to take $\tilde{p}_{m}:=(1,\dotso,1,p,1,\dotso)$, with $p$ in the $m$-th position, for $m\not\in N$. Once again, there exists $n_1\geq m>n_0$ such that $\sup F_{n_1}\not\leq p$. Continuing like this, we see that $\sup F_n\not\leq p$ for infinitely many $n$.
\end{proof}

In the argument presented by \Sz~and Tsaban~\cite{Tsaban}, the topological counterpart of Proposition~\ref{strong.lose.prop} is used to obtain a play $(\mU_0,F_0,\mU_1,F_1,\dotso)$ of an auxiliary instance of the Menger game such that for every $x$ one has $x\in \bigcup F_n$ for infinitely many $n$. This allows one to select $U_m\in F_m$ for each $m\in\omega$ such that the family $\{U_m:m\in\omega\}$ is an open cover for $X$. As it happens with the previous propositions, this argument also has a lattice counterpart.

\begin{lemma}
Let $\PP$ be a pre-Pawlikowski lattice satisfying $\sone(\opV_1,\opV_1)$. Let $(F_n)_{n\in\omega}$ be a sequence of finite subsets of $\PP$ such that for every $p\in\PP\setminus\{1\}$ there are infinitely many $n$ such that $\sup F_n\not\leq p$. Then there are elements $f_n\in F_n$ for every $n$ such that $\sup_{n\in\omega}f_n=1$.
\end{lemma}
\begin{proof}
The proof is the same as in \cite{Tsaban}, up to terminology. For each $n\in\omega$ let
\[\mV_n:=\left\{\bigwedge G:\exists J\in[\omega]^{n+1}\,\forall j\in J\, |G\cap F_j|=1\text{ and }|G|\leq n+1\right\},\]
i.e., the elements of $\mV_n$ are the meets of $n+1$ elements belonging to $n+1$ pairwise distinct sets of the sequence $(F_j)_{j\in\omega}$. The hypothesis about the sequence $(F_n)_{n\in\omega}$ implies that $\sup \mV_n=1$ for every $n\in\omega$: for if $p<1$, there exists a prime element $\tilde{p}<1$ such that $p\leq \tilde{p}$, while the hypothesis gives infinitely many $m$ such that $\sup F_m\not\leq \tilde{p}$; then we may take $J\in[\omega]^{n+1}$ and $f_j\in F_j$ for each $j\in J$ such that $f_j\not\leq \tilde{p}$; since $\tilde{p}$ is prime, it follows that $\bigwedge_{j\in J}f_j\not\leq \tilde{p}$, hence $\bigwedge_{j\in J}f_j\not\leq p$. Therefore, $1\in\PP$ is the only upper bound of $\mV_n$.

Now, since $\sone(\opV_1,\opV_1)$ holds, there exists a sequence $(v_n)_{n\in\omega}$ with $v_n\in\mV_n$ for each $n$, such that $\sup_{n\in\omega}v_n=1$. Notice that $v_0\leq f_{n_0}$ for some $f_{n_0}\in F_{n_0}$, while $v_1\leq f_{n_1}$ for some $f_{n_1}\in F_{n_1}$ such that $n_1\ne n_2$, and so on. Proceeding like this, we obtain $f_{n_k}\in F_{n_k}$ such that $\sup_{k\in\omega}f_{n_k}=1$, so we can pick arbitrary $f_n\in F_n$ for the remaining $n$'s and still get $\sup_{n\in\omega}f_n=1$.
\end{proof}

We finally proceed to the order-theoretic version of Pawlikowski's theorem. However, we must impose a last condition over our lattices. Let us say that $\PP$ is a \textbf{Pawlikowski} lattice if $(\sup A)\wedge b=\sup\{a\wedge b:a\in A\}$ holds for every $b\in \PP$ and every subset $A\subseteq \PP$ having a supreme. This is the order-theoretic version of the identity
\[\left(\bigcup \mU\right)\cap B=\bigcup_{U\in \mU}U\cap B,\]
holding for (open) sets of a topological space $X$.
The distributivity over arbitrary suprema is needed to guarantee that for every $A,B\in\opV_1$, the family $A\wedge B:= \{a\wedge b:a\in A\text{ and }b\in B\}$ is still a member of $\opV_1$, i.e., $\sup A\wedge B=1$.

\begin{theorem}\label{main.theorem2}
Let $\PP$ be a Pawlikowski lattice satisfying $\sone(\opV_1,\opV_1)$. Then $\playerone$ does not have a winning strategy in the game $\gone(\opV_1,\opV_1)$.
\end{theorem}
\begin{proof} The hypothesis over $\PP$ allows one to follow the same arguments of {\Sz} {and} Tsaban~\cite{Tsaban} by simply applying the correspondent order-theoretic versions of the necessary lemmas. The details are left to the reader.
\end{proof}

\section{Further comments}

Although the adaptations of \Sz~and Tsaban simplifications do not require explicit use of points, one could ask how the further assumptions on Pawlikowski lattices affect their proximity to topological spaces. Recall that a lattice $\PP$ is \emph{spatial} if there exists a topological space $(X,\tau)$ such that $\PP$ is a frame isomorphic to $\tau$ in the frame category\footnote{The arrows are the increasing maps preserving arbitrary joins and finite meets, thus including $0$ and $1$.}, where by a frame we mean a complete lattice satisfying the distributive condition of Pawlikowski lattices. Thus, the question rephrases as: is there a non-spatial Pawlikowski lattice? The answer is yes.

Indeed, as implicitly\footnote{It was \emph{pointed} out by Eric Wofsey in this answer:~\url{https://math.stackexchange.com/a/2606878/128988}.} showed by Johnstone in his classic book \emph{Stone Spaces}~\cite{Johnstone}, a lattice $\PP$ is spatial if, and only if, $\PP$ is a \emph{complete} lattice with enough prime elements. On the other hand, our assumption regarding distributivity and prime elements do not yield completeness, as the following example of Eric Worsey shows\footnote{See: \url{https://math.stackexchange.com/a/4019499/128988}.}.

\begin{example}
For an infinite set $X$, let $L:=\{A\subseteq X:A$ is finite or $A$ is cofinite$\}$. One can easily see that $L$ is a non-complete lattice with respect to inclusion, with enough prime elements (the complementary of singletons) and such that \[(\sup \mathcal{A})\cap B=\sup_{A\in \mathcal{A}}(A\cap B)\] for every $B\in L$ and $\mathcal{A}\subseteq L$ with $\sup\mathcal{A}=\bigcup \mathcal{A}\in L$. In other words, $L$ is a non-complete Pawlikowski lattice.
\end{example}

In this sense, the results obtained along this work show that, although a few aspects of points (disguised as prime elements) seem to be necessary to prove Hurewicz's and Pawlikowski's theorems, one does not need the full power of a topology to do so. Therefore, assuming other problems regarding topological games can be rephrased in the lattice language, it is natural to ask the following.

\begin{question}
Do Theorems~\ref{main.theo1} and \ref{main.theorem2} hold for larger classes of lattices?
\end{question}

While the last question implicitly aims for possible applications, another direction could be to search for generalizations of known results in the field. In this case, we mention the following theorem, which extends Lemma 3.12 from García-Ferreira and Tamariz-Mascarúa~\cite{garcia95}.

\begin{theorem}
For a Pawlikowski lattice $\PP$ and a function $f\colon \omega\to\omega$, the following are equivalent:
\begin{enumerate}
\item $\sone(\opV_1,\opV_1)$ holds;
\item for each sequence $(A_n)_{n\in\omega}$ where $A_n\in\opV_1$ for every $n$, there are finite subsets $B_n\in [A_n]^{<f(n)+1}$ such that $\bigcup_{n\in\omega}B_n\in \opV_1$.
\end{enumerate}
\end{theorem}

As a final remark, we notice that Hurewicz property $\mathsf{U}_{\operatorname{fin}}(\mO,\Gamma)$ can also be stated for lattices: we may say that a bounded lattice $\PP$ is \emph{Hurewicz} if for every sequence $(A_n)_{n\in\omega}\in (\opV_1)^\omega$ there are finite subsets $F_n\subseteq A_n$ for each $n$ such that 
\[\bigvee\left\{\bigvee F_n:n\in\omega\right\}=1\]
with $\left\{n:\bigvee F_n\not\leq p\right\}$ cofinite for every prime element $p\in\PP$. The translation of this property to the lattice context suggests even further investigations. 

\section*{Acknowledgements}

I wish to express my gratitude to Dione Andrade Lara, Rodrigo ``Rockdays'' Dias and Vinicius Rodrigues for the fruitful discussions about the results presented here.



\bibliography{ccc}{}
\bibliographystyle{abbrv}

\end{document}